\documentclass[12pt,a4paper]{amsart}

\usepackage[margin=1in]{geometry}
\usepackage{mathtools}
\usepackage{graphicx}

\newtheorem{theorem}{Theorem}
\newtheorem{lemma}{Lemma}
\newtheorem{claim}{Claim}
\newtheorem{remark}{Remark}

\begin{document}
\title{Growth of Replacements}
\author{Vuong Bui}
\address{Vuong Bui, Institut f\"ur Informatik, 
Freie Universit\"{a}t Berlin, Takustra{\ss}e~9, 14195 Berlin, Germany}
\thanks{The author is supported by the Deutsche 
Forschungsgemeinschaft (DFG) Graduiertenkolleg ``Facets of Complexity'' 
(GRK 2434).}
\email{bui.vuong@yandex.ru}
\begin{abstract}
	The following game in a similar formulation to Petri nets and chip-firing games is studied: Given a finite collection of baskets, each has an infinite number of balls of the same value. Initially, a ball from some basket is chosen to put on the table. Subsequently, in each step a ball from the table is chosen to be replaced by some $2$ balls from some baskets. Which baskets to take depend only on the ball to be replaced and they are decided in advance. Given some $n$, the object of the game is to find the maximum possible sum of values $g(n)$ for a table of $n$ balls.

	In this article, the sequence $g(n)/n$ for $n=1,2,\dots$ will be shown to converge to a growth rate $\lambda$. Furthermore, this value $\lambda$ is also the rate of a structure called pseudo-loop and the solution of a rather simple linear program. The structure and the linear program are closely related, e.g. a solution of the linear program gives a pseudo-loop with the rate $\lambda$ in linear time of the number of baskets, and vice versa with the pseudo-loop giving a solution to the dual linear program. A method to test in quadratic time whether a given $\lambda_0$ is smaller than $\lambda$ is provided to approximate $\lambda$. When the values of the balls are all rational, we can compute the precise value of $\lambda$ in cubic time, using the quadratic time rate test algorithm and the binary search with a special condition to stop. Four proofs of the limit $\lambda$ are given: one just uses the relation between the baskets, one uses pseudo-loops, one uses the linear program and one uses Fekete's lemma (the latest proof assumes a condition on the rule of replacements).
\end{abstract}

\maketitle

\section{Introduction}
\label{sec:introduction}
Suppose we have a finite number of baskets, each basket contains infinitely many balls of the same value. We start with choosing a ball from some basket to put on a table. At each subsequent step, we replace one ball on the table by two balls from some baskets with respect to a given set of rules that only involves the baskets where the balls are from. When there are $n$ balls on the table for a given $n$, we stop and evaluate the sum (and the average) of the values of all the $n$ balls. Our aim is to achieve the highest possible sum (and average) for a given $n$ by choosing appropriately the basket of the first ball to put on the table and the ball to replace at each subsequent step. An asymptotic behavior is that when $n$ tends to infinite, this best average converges to a constant $\lambda$, which is called the \emph{growth rate} of the system.

Figure \ref{fig:example} gives an example of the setting. There are four baskets with balls of values $1,2,3,4$. The replacements are done in a rotating manner. 
\begin{figure}[ht]
    \centering
    \includegraphics[width=0.5\textwidth]{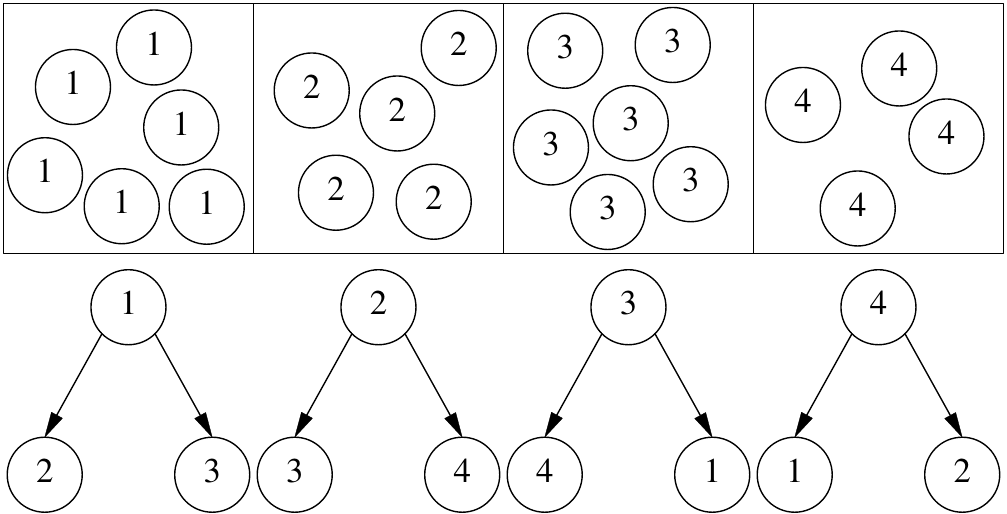}
    \caption{A simple example}
    \label{fig:example}
\end{figure}

Let us state the problem in an equivalent but more formal way, which will be used throughout the text. The formulation starts with a collection $V$ of functions $v: \mathbb N^+\to \mathbb R$. Denote $c_v = v(1)$ for each function $v\in V$, which will be called the \emph{starting values} later. Assume we have an assignment of a pair of functions $M(v) = (u, w)$ to each $v$ ($u,v,w\in V$ not necessarily different) such that $v(n)$ for $n\ge 2$ is given by
\begin{equation}
	\label{eq:setting}
	v(n) = \max_{1\le m\le n-1} u(n-m) + w(m).
\end{equation}

Let $g(n)$ denote the maximum of the values of the functions at $n$, that is
\[
	g(n) = \max_{v\in V} v(n).
\]
 
In this text, we show that the sequence $\{g(n)/n\}_{n=1}^\infty$ converges to the so-called \emph{growth rate $\lambda$ of the system}:
\[
	\lambda = \lim_{n\to\infty} \frac{g(n)}{n}.
\]

The equivalence between the formulations is not so hard to see. Each function $v$ corresponds to a basket with $c_v$ as the value of a ball in the basket. The value of $v(n)$ is the maximum sum obtained from $n$ balls if we start with a ball from the basket corresponding to $v$. The value of $g(n)$ is then the maximum sum when we do not restrict which ball to start with.

The following table provides some beginning values of $g(n)$ for the introductory example. It suggests that the growth rate is some number that starts with $1.6$. In fact, the growth rate is $11/3=1.6666\dots$, by the mechanism in Figure \ref{fig:pattern}.
\begin{center}
\begin{tabular}{ |c|c|c|c|c|c|c|c|c|c|c|c|c|c|c|c|c|c|c| }
\hline
$n$ & 1 & 2 & 3 & 4 & 5 & 6 & 7 & 8 & 9 & 10 & 11 & 12 & 13 & 14 & 15 & 16 & 17 & 18 \\
\hline
$g(n)$ & 4 & 7 & 10 & 14 & 18 & 21 & 25 & 29 & 32 & 36 & 40 & 43 & 47 & 51 & 54 & 58 & 62 & 65 \\
\hline
\end{tabular}
\end{center}

\subsection*{Related notions}
The readers who are familiar with Petri nets \cite{peterson1977petri} and chip-firing games \cite{bjorner1992chip} may recognize that the problem is somewhat like both of them in the setting but different in the object. We can formulate the problem as a Petri net: There is a place corresponding to each basket. For each rule of replacing a ball from basket $A$ by $2$ balls from baskets $B$ and $C$, we establish a transition that takes the place of $A$ as the input and the places of $B,C$ as the outputs. An initial place is chosen to put a token into, and we start firing the transitions. We can also formulate the problem as a variant of the chip-firing game: Consider a directed graph with loops and multiple edges allowed and the outdegree of every vertex is $2$. We start the game by putting a chip on some vertex. In each subsequent step, we choose a chip at some vertex $v$, remove it from the graph and add a chip in each of the two vertices that can be reached from $v$ by a directed edge. (In the original chip-firing game, we need to remove $2$ chips from vertex $v$ in each step, which is the outdegree.) In either formulation, the value of a token (a chip) depends only on the place (the vertex) it lies on. The object is to find a way to obtain precisely $n$ tokens ($n$ chips) for a given $n$ so that the sum of the values is maximum. The asymptotic behavior of the maximum sum with respect to $n$ is to be investigated. The problem is more similar to Petri net in the way tokens are transited in each firing and is more similar to the chip-firing game in the number of inputs for each firing. However, we fix the number of outgoing edges in both formulations as the convergence may not hold otherwise. For example: Suppose there is a graph of two vertices $A,B$ with two loops around $A$ and three loops around $B$. If the value of a chip in $A,B$ is respectively $5,10$, then the maximum sum is $10n$ for $n$ odd and $5n$ for $n$ even, a divergence. Also, there may be no way to obtain $n$ chips for certain $n$ in some other settings. Note that if we study the growth of the sum of values with respect to the number of \emph{steps} (instead of the number of balls), the growth rate, which is $20$, is well defined in the example.

The problem we are studying is in some sense of the same type as Fekete's lemma \cite{fekete1923verteilung}, which states that for a \emph{superadditive} function $f: \mathbb N^+ \to \mathbb R$, that is $f(n+m) \ge f(n) + f(m)$ for any $n,m\ge 1$, we have $\lim_{n\to\infty} f(n)/n$ exists. Our setting differs from Fekete's lemma in two points: (i) instead of the equivalent inequality $f(n) \ge \max_{1\le m\le n-1} f(n-m) + f(m)$, we use the equality as in Equation \eqref{eq:setting}, and (ii) instead of one function, a collection of functions are involved. Note that if the equality in Equation \eqref{eq:setting} is replaced by the inequality, then the limit we are studying does not necessarily exist. For example, consider the functions $v_0, v_1$ so that $v_0(n)\ge\max_{1\le m\le n-1} v_0(n-m) + v_0(m)$ and $v_1(n)\ge\max_{1\le m\le n-1} v_0(n-m) + v_0(m)$. If the two functions are $v_0(n)=n$ for every $n$ and $v_1(n)=n$ for odd $n$ and $v_1(n)=2n$ for even $n$, then the maximum average changes between $1$ and $2$ as $n$ increases. However, if the dependency graph, which will be defined later, is connected, then the limit still exists, by the proof in Section \ref{sec:proof-by-fekete}. The readers can check for themselves that the techniques there also work for the case of inequalities.

One can also formulate this problem in terms of context free grammar (in Chomsky normal form). Let us consider the following language: There is a nonterminal symbol $V$ associated to each function $v$, the production rule $V\rightarrow UW$ corresponds to the assignment of $u,w$ to each $v$, and there is also a production $V\rightarrow \underline{v}$ for each nonterminal symbol $V$, where $\underline{v}$ is a terminal symbol for which we assign the weight $v(1)$. We define the weight of a word to be the sum of all the symbols in that word. The function $v(n)$ is then the maximum weight of a word of length $n$ if we start with the symbol $V$. Every other term is mapped accordingly.

In fact, the original motivation is that the problem is a special case of a problem posed by Rote in \cite{rote2019maximum}: Given a bilinear map $*:\mathbb R^d\times \mathbb R^d\to\mathbb R^d$ and a vector $v\in\mathbb R^d$ (the coefficients of $*$ and the entries of $v$ are all nonnegative), we consider all the possible ways to combine $n$ instances of $v$ using $n-1$ instances of $*$. For a given $n$, we are interested in the largest entry of a resulting vector one can achieve, and the question is whether this follows a growth rate, and how does it follow if so? In that paper, the maximum number of minimal dominating sets in a tree of $n$ leaves is studied in this way. Many applications to other quantities of trees are presented in \cite[Section 5]{rosenfeld2021growth}. However, our instance in this paper is not general enough to cover the applications there. (Note that in order to see our formulation as a special case to the general setting, one needs to apply logarithms, which turn products into sums.) The limit presenting the exponential growth in the general setting is proved to be valid for the case of positive vectors $v$ in \cite{bui2021growth} by a special structure called \emph{linear pattern}. It means that the growth rate for our problem is also valid. However, in this work we give different proofs to our problem since it is somewhat simpler. In fact, we will use a structure called \emph{pseudo-loop}, which is very similar to linear pattern. When $v$ is only nonnegative, the limit may not exist. Checking if the limit superior is a given number is shown to be undecidable in \cite{rosenfeld2022undecidable}. Meanwhile, the growth rate in our setting can be computed precisely, and quite efficiently in certain cases.

\subsection*{Growth rate under different perspectives}
Four proofs for the validity of the limit $\lambda$ will be given. The first one uses only the dependency graph (defined later) as in Theorem \ref{thm:lim_exists} below. The second one as in Theorem \ref{thm:supremum} relates the growth rate to the rates of pseudo-loops. The third one as in Theorem \ref{thm:equiv_to_linear_prog} relates the growth rate to the solution of a linear program. The fourth one, which assumes the connectedness of the dependency graph, is given in Section \ref{sec:proof-by-fekete}. Although the latest proof does not work without the condition, it demonstrates a nice application of Fekete's lemma. Among the proofs, the proof that removes inner pseudo-loops in Section \ref{sec:over_all_loops} is perhaps the simplest and shortest one.

We will regard not only the growth rate of the system but also the \emph{growth rate of an individual function} $\lambda_v = \lim_{n\to\infty} v(n)/n$ sometimes. It turns out that both the growth rate of the system and the growth rate of each individual function are valid due to the following theorem, whose proof is given in Section \ref{sec:lim_exists}.

\begin{theorem}
	\label{thm:lim_exists}
	Both $\lambda = \lim_{n\to\infty} g(n)/n$ and $\lambda_v = \lim_{n\to\infty} v(n)/n$ for every $v$ exist.
\end{theorem}

The proof of Theorem \ref{thm:lim_exists} and further study rely on the following definitions: the dependency graph and composition trees.

The \emph{dependency graph} is the graph whose set of vertices is $V$ and there is a directed edge from $v$ to $u$ if and only if one of the two functions in $M(v)$ is $u$ (loops are allowed).
As the dependency graph is directed, it can be partitioned into strongly connected components. A component is said to be a \emph{single component} if it contains only one vertex and there is no loop for that vertex. In other words, the only vertex $v$ in a single component has the outgoing edges $vu$ and $vw$ for $u,w$ both different from $v$. One can see that being single for a component is identical to acyclicity (by definition, a component of one vertex that has a loop is not a single component).
Let us consider the \emph{condensation} of the dependency graph, which is the acyclic graph with each vertex corresponding to a strongly connected component and there is a directed edge $UV$ if and only if there is an edge $uv$ with $u\in U$ and $v\in V$ in the dependency graph. The condensation defines a partial order between components where a directed edge $UV$ means $U\ge V$. A minimal component, which is not greater than any other component, cannot be a single component since otherwise the minimality implies the only function depending on itself, which in turn contradicts its membership in the single component.

One can relate the evaluation of a function $v(n)$ to \emph{composition trees} whose definition is given as follows. For any binary tree of $n$ leaves, we start with labeling the root of the tree with $v$ and suppose $M(v)=(u,w)$, we label the left child of the root with $u$ and the right child with $w$. We subsequently label all the vertices of the subtrees with the same method. Suppose the labels of the leaves are $v_1, \ldots, v_n$, then the tree would be evaluated as $\sum_i c_{v_i}$. Such a labeled tree with this way of evaluation is called a \emph{composition tree}. One can see that the value of $v(n)$ is the largest evaluation over all the composition trees of $n$ leaves. Note that the label for a vertex in a composition tree is actually a vertex of the dependency graph. Unless stated otherwise, all the trees will be regarded as composition trees.

Let us consider a simple pattern for composition trees. Let $T$ be a tree with some label for the root and a specially marked leaf that has the same label as the root. Let the sequence of trees $\{T^n\}_{n=1}^\infty$ be defined so that $T^1=T$ and $T^n$ for $n\ge 2$ is obtained from $T^{n-1}$ by replacing the marked leaf of $T^{n-1}$ by $T$. The marked leaf of $T^n$ is defined to be the marked leaf of the instance of $T$. A tree $T$ defined in this way is called a \emph{pseudo-loop}. The path from the root to the marked leaf is called the \emph{main path}. The \emph{value of a pseudo-loop} is defined to be the sum of all leaves excluding the leaf at the main path. It is not hard to see that the evaluation of the trees $\{T^n\}_n$ follows a rate, which is the average of the values of all the leaves excluding the marked one. This rate will be called \emph{the rate of the pseudo-loop}.

For the pseudo-loop in Figure \ref{fig:pattern}, the marked leaf and the root both have label $2$. The value of the pseudo-loop is $4+4+3=11$, which means the rate of the pseudo-loop is $11/3$.

The following definitions on pseudo-loops will be also used later. For a subtree with the root $a$ having some label and one of its descendants $b$ having the same label as $a$, a pseudo-loop obtained from the subtree by removing every further descendant of $b$ is called an \emph{inner pseudo-loop}. By \emph{removing} an inner pseudo-loop we mean contracting the whole inner pseudo-loop into a vertex. If removing an inner pseudo-loop from a pseudo-loop still gives a valid pseudo-loop, then the inner pseudo-loop is said to be \emph{removable}. Note that removability is considered only in the context of a pseudo-loop while an inner pseudo-loop can be a subgraph of either a tree or a pseudo-loop.

The relation between the growth rate $\lambda$ and the notion of pseudo-loop is given in the following theorem.
\begin{theorem}
	\label{thm:supremum}
	The growth rate exists and it is the supremum of the rates of all pseudo-loops.
\end{theorem}

Furthermore, we can find the best rate in a finite set of pseudo-loops. That is to say the supremum is always attainable.
\begin{theorem}
	\label{thm:maximum}
	There exists a pseudo-loop with the same rate as the growth rate of the system. It can be found among pseudo-loops that do not contain any removable inner pseudo-loop. In particular, such a pseudo-loop has at most $|V|2^{|V|-1}$ leaves after excluding the marked one.
\end{theorem}

The proofs of the two above theorems can be found in Section \ref{sec:over_all_loops}. A pseudo-loop that attains the growth rate $11/3$ of the introductory example is given in Figure \ref{fig:pattern}. Note that the space of the pseudo-loops in Theorem \ref{thm:maximum} is so large that finding a solution for a large $|V|$ is impractical when using only the brute force search.

\begin{figure}[ht]
    \centering
    \includegraphics[width=0.16\textwidth]{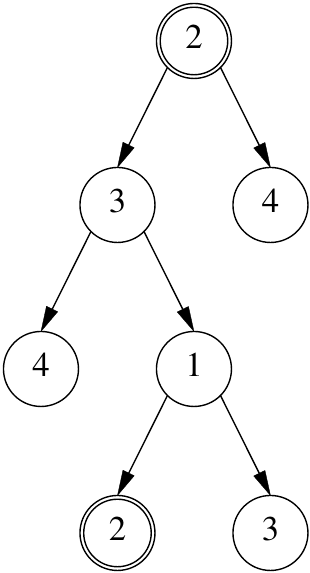}
    \caption{A pseudo-loop attaining the growth rate}
    \label{fig:pattern}
\end{figure}

The readers may relate pseudo-loops to linear patterns in the work \cite{bui2021growth}, where a similar result to Theorem \ref{thm:supremum} is given. However, the growth rate in \cite{bui2021growth} is not always the rate of a linear pattern as in the specific case of this paper.

In order to make another equivalence to the growth rate than the maximum rate over all pseudo-loops, we study the following system \footnote{This system was suggested by G\"unter Rote (private communication). The readers may relate it to \cite[Proposition 5.1]{rote2019maximum}.} of $2|V|$ inequalities: For every function $v$,
\begin{equation}
	\label{eq:ineq_sys}
	\begin{matrix}
		z_v &\ge  &c_v - \theta\\
		z_v &\ge  &z_u + z_w
	\end{matrix}
\end{equation}
where $\{z_v: v\in V\}$ and $\theta$ are variables, $c_v$ is $v(1)$ as already defined, and $M(v) = (u, w)$.

The set of the solutions is nonempty, e.g. $z_v=0$ for all $v$ and $\theta$ is the maximum of all $c_v$.

Consider the linear program minimizing $\theta$ subjecting to System \eqref{eq:ineq_sys}, we have the following representation of the growth rate.
\begin{theorem}
	\label{thm:equiv_to_linear_prog}
	The growth rate exists and it is the solution of $\theta$ to the linear program.
\end{theorem}

A proof is given in Section \ref{sec:equiv_to_linear_prog}. The linear program has $|V|+1$ variables and $2|V|$ inequalities but it is still rather simple and actually quite resembles the setting of the problem. In fact, given a solution of all variables, one can construct a pseudo-loop with the growth rate $\lambda$ in linear time of the number of functions (variables) by a method provided during the course of the proof. A more precise order of $g(n)$ is also shown there: The difference $g(n)-n\lambda$ is bounded. A corollary of this fact is a small interval bounding $\lambda$ provided the value of $g(n)$ for an $n$ large enough. 

Also in Section \ref{sec:equiv_to_linear_prog}, we consider the dual linear program. An interesting point is that given a pseudo-loop of the rate $\lambda$, we can give a solution to the dual program in linear time of the variables. Moreover, the construction is more straightforward than the other direction with the original program.

We now consider some computational aspects of the growth rate.
\begin{theorem}
	\label{thm:rate-test}
	Given any proposal $\lambda_0$, one can decide if $\lambda_0<\lambda$ in quadratic time of the number of functions. 
\end{theorem}
Theorem \ref{thm:rate-test} immediately gives a reasonable algorithm of approximating $\lambda$ with the time complexity $O(-|V|^2 \log \epsilon)$ for a given precision $\epsilon$. In the proof in Section \ref{sec:rate_test}, the readers will find that $\lambda_0\ge\lambda$ is equivalent to whether each function $v$ has a maximum value of $v(n)$ over all $n$ when the considered system uses the value $c_u-\lambda_0$ instead of $c_u$ for every $u$. Such maximum values are also computed as a by-product. When $\lambda_0=\lambda$, these maximum values turn out to be a solution of the linear program for $\theta=\lambda$. It means the value of $\lambda$ alone can give a solution (of other variables) to the linear program, which in turn gives a pseudo-loop of the same rate.

\begin{theorem}
	\label{thm:cubic-time}
	When the starting values are all rational, we can compute the growth rate precisely in cubic time of the number of functions.
\end{theorem}
Theorem \ref{thm:cubic-time} is a combination of the results in Theorem \ref{thm:maximum} and Theorem \ref{thm:rate-test}. The idea is that the growth rate is a fraction with the denominator not too big, therefore, one can stop the binary search when the interval is small enough. Details are given in Section \ref{sec:cubic-time}. How to compute the growth rate efficiently in case the starting values are not necessarily rational, such as $\pi,e,\dots$, is still open. Note that the straight algorithm by Theorem \ref{thm:maximum} may take double exponential time.

As for some final remarks on possible extensions, inspired by the context free grammar, one may extend the setting to allow nondeterministic replacements in the sense that one may choose among several rules to apply to a ball. The readers can check that the approach still works in the new setting. Although as pointed out before that the replacement by more than two balls may result in divergence, we can still obtain convergence if we consider $n$ as the number of steps instead of the number of balls, by the same approach. However, the approach may not work in the case of an infinite $V$. This direction opens a large room for further research.

\section{Growth rate in terms of the functions}
\label{sec:lim_exists}
To prove the growth rate $\lambda$ and $\lambda_v$ for every function $v$ exist, we give first the following lemma, which is kind of in the same spirit as Fekete's lemma, and should be of its own interest. 

For convenience, in the statement of the lemma and in the proof, all the integers that are supposed to be used for indexing functions $u,v$ will be treated as elements in $\mathbb Z / k\mathbb Z$. In particular, it is the case of the indices $i, i^*,j$. 

\begin{lemma}
	\label{lem:round_fekete}
	Given $2k$ ($k\ge 1$) functions $v_0, \dots, v_{k-1}, u_0, \dots, u_{k-1}: \mathbb N^+\to \mathbb R$ such that for every $0\le i\le k-1$ and every $n\ge 2$,
	\[
		v_i(n) = \max_{1\le m\le n-1} v_{i+1}(n-m) + u_{i+1}(m).
	\]
	Then for every $i$,
	\[
		\lim_{n\to\infty} \frac{v_i(n)}{n} = \sup_{m_0\ge 1, \dots, m_{k-1}\ge 1} \frac{\sum_{j=0}^{k-1} u_j(m_j)}{\sum_{j=0}^{k-1} m_j}.
	\]
\end{lemma}

\begin{proof}
	Denote by $R$ the value of the supremum (note that it can be infinite). To prove the theorem, it suffices to verify the following two points for every $i$:

	(i) $\liminf_{n\to\infty} v_i(n)/n \ge R$.

	By the definition of $R$, for any $R'<R$, there are $m_0, \dots, m_{k-1}$ such that
	\[
		\frac{\sum_{j=0}^{k-1} u_j(m_j)}{\sum_{j=0}^{k-1} m_j} > R'.
	\]
	Let $m=m_0+\dots+m_{k-1}$. For every $n$, if $n=mt+p$ for some integer $t$ and $1\le p\le m$, we have the lower bound $v'_i(n) \le v_i(n)$ with
	\[
		v'_i(n) = v_i(p) + t\left(\sum_{j=0}^{k-1} u_j(m_i)\right).
	\]

	Since $v_i(p)$ is bounded, the sequence $\{v'_i(n) / n\}_n$ tends to $\left(\sum_{j=0}^{k-1} u_j(m_j)\right)/\left(\sum_{j=0}^{k-1} m_j\right)$, which is greater than $R'$. It follows that $\liminf_{n\to\infty} v_i(n)/n > R'$ for any $R'<R$, which implies $\liminf_{n\to\infty} v_i(n)/n \ge R$.

	(ii) $\limsup_{n\to\infty} v_i(n) / n \le R$ (we assume $R\ne\infty$ otherwise it is trivial).
	
	Assume $\limsup_{n\to\infty} v_i(n)/n = R' > R$, we will show a contradiction by giving $m_0\ge 1, \dots, m_{k-1}\ge 1$ so that
	\[
		\frac{\sum_{j=0}^{k-1} u_j(m_j)}{\sum_{j=0}^{k-1} m_j} > R.
	\]

	For each $i$ and $n$, due to the evaluation of $v_i(n)$, there exist a number $t$ and $tk$ numbers $m_j^{(s)}$ for $0\le j\le k-1$, $1\le s\le t$ such that: $\sum_{j,s} m_j^{(s)} = n-1$, all of them are nonzero except possibly $m_{i^*}^{(t)}, m_{i^*+1}^{(t)}, \dots, m_{i}^{(t)}$ for some $i^*$ (if there is no zero, we let $i^*=i+1$), and
	\[
		v_i(n) = v_{i^*-1}(1) + \sum_{s=1}^t \sum_{j=0}^{k-1} u_j(m_j^{(s)}),
	\]
	where $u_j(0)$ is assumed to be zero for every $j$. (The number $t$ can be understood as the number of rounds.)

	Let ${m'}_j^{(s)}=m_j^{(s)}$, but we set ${m'}_{i^*}^{(t)} = {m'}_{i^*+1}^{(t)} = \dots = {m'}_{i}^{(t)} = 1$ if there are corresponding zeros in $\{m_j^{(s)}\}$, we have 
	\begin{equation}
		\label{eq:new_eval}
		\sum_{s=1}^t \sum_{j=0}^{k-1} u_j({m'}_j^{(s)}) = v_i(n) - v_{i^*-1}(1) + \sum_{j=i^*}^i u_j(1).
	\end{equation}

	By the definition of $R'$, for every $\epsilon>0$, there is an arbitrarily large $n$ such that  
	\[
		\frac{v_i(n)}{n} > R'-\epsilon.
	\]

	Note that the right hand side of Equation \eqref{eq:new_eval} is the sum of $v_i(n)$ and a bounded sum, and the difference between the sum of all ${m'}_j^{(s)}$ and the sum of all $m_j^{(s)}$ is also bounded. It means that for every $\epsilon'>0$, we can choose a small enough $\epsilon$ and a large enough $n$ such that
	\[
		\sum_{s=1}^t \sum_{j=0}^{k-1} u_j({m'}_j^{(s)}) > (R'-\epsilon')\left(\sum_{s=1}^t \sum_{j=0}^{k-1} {m'}_j^{(s)}\right).
	\]
	This is followed by the existence of some $s^*$ such that
	\[
		\frac{\sum_{j=0}^{k-1} u_j({m'}_j^{(s^*)})}{\sum_{j=0}^{k-1} {m'}_j^{(s^*)}} > R'-\epsilon'.
	\]

	Since $\epsilon'$ can be arbitrarily small, $R'-\epsilon' > R$ for some $\epsilon'$, and since all ${m'}_j^{(s^*)}\ge 1$, we have a contradiction with the supremum $R$.

	By (i) and (ii), the conclusion follows.
\end{proof}

Now we can prove Theorem \ref{thm:lim_exists}.

Consider the partial order between the strongly connected components of the dependency graph, as already pointed out in Section \ref{sec:introduction}, the minimal component cannot be a single component. Therefore, each function in a minimal component should be in a cycle and the existence of its growth rate is confirmed by Lemma \ref{lem:round_fekete}. Consider a non-minimal component with the assumption that we already have growth rates for the functions in all smaller components. If the considered component is not single, then every function has a growth rate as already reasoned. In the other case, the only function $v$ in the component has $M(v)=(u,w)$ with $u,w$ from smaller components, hence they already have growth rates by induction hypothesis. Since $v(n)=\max_m u(n-m) + w(m)$, the larger rate of $u$ and $w$ is the growth rate of $v$. By induction, all functions have growth rates. It follows from $g(n) = \max_v v(n)$ that $g(n)$ also has a growth rate, which is the largest rate over all the functions $v$.

\begin{remark}
	Although Lemma \ref{lem:round_fekete} also covers the case the limit is infinite, the limits in our application are obviously finite since the value $v(n)/n$ for any function $v\in V$ is always contained in the range of the minimum and maximum starting values.
\end{remark}

\section{Growth rate as the maximum rate over all the pseudo-loops}
\label{sec:over_all_loops}
\subsection*{Growth rate as the supremum rate}
We prove Theorem \ref{thm:supremum} in different ways, one removes inner pseudo-loops while the other extends a tree to a pseudo-loop.

At first, it is obvious that $\liminf_{n\to\infty} g(n)/n \ge \sup_T \lambda_T$, where $\lambda_T$ is the rate of a pseudo-loop $T$.
Indeed, consider a pseudo-loop $T$ and let $q$ be the number of leaves of $T$ excluding the marked one. For every $n$, let $n$ be expressed as $n=pq+r$ for an integer $p$ and $1\le r\le q$, it can be seen that $g(n)\ge pq\lambda_T + O(1)$ by considering the tree obtained from $T^p$ by replacing the marked leaf of $T^p$ by any tree of $r$ leaves. The corresponding lower bound of $g(n)/n$ converges to $\lambda_T$, the conclusion follows.

Let $\bar\lambda=\limsup_{n\to\infty} g(n)/n$ and $\lambda^*=\sup_T \lambda_T$, it remains to prove that
\[
    \bar\lambda\le\lambda^*.
\]

\begin{proof}[Proof of Theorem \ref{thm:supremum} that removes inner pseudo-loops]
	Assume the contrary that $\bar\lambda > \lambda^*$, we give a contradiction by the existence of a pseudo-loop with a higher rate than $\lambda^*$.

An inner-pseudo-loop-free tree has a bounded number of leaves. In other words, any tree of many enough leaves has an inner pseudo-loop. Subsequently removing all inner pseudo-loops results in an inner-pseudo-loop-free tree. The value of the original tree is the sum of the values of all removed inner pseudo-loops and the new tree. 

	By the definition of $\bar\lambda$, for every $\epsilon>0$ and any $N_0$, there exists some $N>N_0$ so that $g(N)/N>\bar\lambda - \epsilon$.

	Choose some $\epsilon$ small enough and consider such a large $N$.
As the tree has the value at least $N(\bar\lambda-\epsilon)$, we have the sum of the values of all the removed inner pseudo-loops is $N(\bar\lambda-\epsilon) - O(1)$, where $O(1)$ is the value of the new tree.

Since the total number of leaves of the pseudo-loops is $N-O(1)$, there must be a pseudo-loop of rate at least the average
\[
	\frac{N(\bar\lambda-\epsilon)-O(1)}{N-O(1)}.
\]

When $N$ is large enough and $\epsilon$ is small enough, the above average is arbitrarily close to $\bar\lambda$, hence greater than $\lambda^*$, contradiction.
\end{proof}

The other proof is a bit sketchy as follows.
\begin{proof}[Proof of Theorem \ref{thm:supremum} that extends a tree to a pseudo-loop]
	If there is a path from $u$ to $v$ then there is a composition tree $T(u,v)$ of a bounded number of leaves (and value) so that the root is labeled $u$ and one of the leaves is labeled $v$.

	If $g(N)>N(\bar\lambda-\epsilon)$ corresponds to a tree of $N$ leaves with the root labeled $v$ and a leaf labeled $u$ so that $u,v$ are in the same component, then replacing the leaf by $T(u,v)$, we obtain a pseudo-loop with the rate at least
	\begin{equation} \label{eq:root-leaf-same-comp}
		\frac{N(\bar\lambda-\epsilon)+O(1)}{N+O(1)},
	\end{equation}
	which is greater than $\lambda^*$ when $N$ is large enough and $\epsilon$ is small enough.

	If no leaf has the label in the same component as the label $v$ of the root, we consider a subtree $T'$ of $T$ such that $|T|/3\le |T'|\le 2|T|/3$, where $|T|$ is the number of leaves of $T$.
	The value of $T'$ is at most $|T'|(\bar\lambda+\epsilon)$ when we choose $N$ large enough. If the root of $T'$ has the label in the same component as $v$, then we have the same situation as in \eqref{eq:root-leaf-same-comp}. Indeed, the value of the tree $T_0$ obtained from $T$ by contracting $T'$ into a single leaf would have the value at least
	\[
		N(\bar\lambda-\epsilon) - |T'|(\bar\lambda+\epsilon) + O(1).
	\]

	If the label of the root of $T'$ is in a lower component than the component of $v$, then we have the same problem for $T'$ with the labels of the vertices being in one less components than $T$ and the value of $T'$ at least 
	\[
		N(\bar\lambda-\epsilon) - |T_0|(\bar\lambda+\epsilon) + O(1),
	\]
	since the value of $T_0$ is at most $|T_0|(\bar\lambda+\epsilon)$.

	Recursively treating smaller problems with $N$ large enough and $\epsilon$ small enough would give a situation where there is a leaf having the label in the same component as the label of the root.
\end{proof}

In the latter proof, there is another way to treat the situation where there is no leaf having the label in the same component as the root. It is done by the following lemma.

\begin{lemma}
	\label{lem:bounded-partition}
	Let $G$ be the dependency graph and $A$ any composition tree. Then there exist a bounded number of disjoint subtrees of $A$ such that they cover all leaves of $A$ and in each subtree the label of the root is in the same component as the label of a leaf. In particular, a tight bound is $2^c$ where $c$ is the number of single components in $G$.
\end{lemma}

\begin{proof}
	Consider any subtree with the label of the root not in a single component of $G$. One of the two children must have the label in the same component as the root. If this child is a leaf, then the subtree has the root and one leaf having the labels in the same component. Otherwise, we still have the same situation by recursively following this vertex and its children until we see a leaf having the label in the component. The process will eventually finish due to the finiteness and acyclicity of the tree.

	The above consideration of the root is indeed the case when $c=0$, and one subtree is sufficient to cover as reasoned above. If $c>0$, it maybe the case that the label of the root is in a single component of $G$, then the number of subtrees required is the sum of those numbers in the left branch and the right branch of the root. Since the label in a single component of $G$ cannot be revisited, each branch can have vertices of the labels from the remaining $c-1$ single components only (beside other nonsingle components). By recursively following them, we need at most $2^c$ subtrees to cover. 

	Although we just need this number to be bounded for later usage, this bound is actually tight. For example, let $G$ have $k+1$ vertices $v_0,v_1,\dots, v_k$ with the edges $v_i v_{i+1}$ for $0\le i\le k-1$ and the loop $v_k v_k$. The number of single components in $G$ is $k$. Let $A$ be the perfect binary tree of height $k+1$ and the root (say, at depth $0$) is labeled $v_0$. It follows that the vertices at depth $i$ for $0\le i\le k$ are labeled $v_i$ and the vertices at depth $k+1$, which are all leaves, are labeled $v_k$. One can see that $2^k$ subtrees are needed to cover all the leaves of $A$.
\end{proof}

By the lemma, the leaves of $T$ are covered by some $k$ disjoint subtrees $T_1,\dots,T_k$ for a bounded $k$ so that a leaf in each subtree has the label in the same component as the root. The value of $T$ is the sum of the values of those trees. For each tree $T_i$ with the root labeled $v_i$ and the leaf labeled $u_i$, we transform it to $T'_i$ by replacing the leaf labeled $u_i$ by the tree $T(u_i,v_i)$. The new tree $T'_i$ can be seen as a pseudo-loop, whose value is denoted by $f(T'_i)$. Note that the difference in value and in number of leaves between $T_i$ and $T'_i$ is bounded. That is 
\[
	f(T'_1)+\dots+f(T'_k) \ge N(\bar\lambda -\epsilon) + O(1),
\]
where the quantity $O(1)$ is due to the boundedness of $k$. As the number of leaves (excluding the marked ones) in the pseudo-loops $T'_i$ is $(|T'_1|-1)+\dots+(|T'_k|-1)=N+O(1)$, there is an $i$ so that
\[
	\frac{f(T'_i)}{|T'_i|-1} \ge \frac{N(\bar\lambda -\epsilon) + O(1)}{N+O(1)}.
\]

Note that the left hand side is the rate of $T'_i$. When $\epsilon$ is small enough and $N$ is large enough, the right hand side is greater than $\lambda^*$, contradiction.

\subsection*{Growth rate as the maximum rate}
Although the space of all pseudo-loops is infinite and the supremum of the rates may not belong to any particular pseudo-loop, we show that the latter is not the case by the fact that we just need to look into the set of pseudo-loops that do not contain any removable inner pseudo-loop to find one with the best rate. In other words, we prove Theorem \ref{thm:maximum}, as follows.

\begin{proof}[Proof of Theorem \ref{thm:maximum}]
	In order to prove the theorem, it suffices to show that any pseudo-loop containing a removable inner pseudo-loop does not need to be considered in the sense that there exists a pseudo-loop of fewer leaves with at least that rate. In other words, the space of pseudo-loops to be considered is finite.

	Indeed, if the inner pseudo-loop has a lower or equal rate to the original one, then removing the former does not reduce the rate of the latter. If the inner one has a higher rate, then that inner one itself is a pseudo-loop with a higher rate. In both cases, we can ignore the original pseudo-loop.

	It remains to show that a pseudo-loop without any removable inner pseudo-loop has at most $|V|2^{|V|-1}$ leaves after excluding the marked one. On the main path from the root to the marked leaf, the subpath from the vertex following the root to the marked leaf should not have two vertices of the same label, otherwise we have a removable inner pseudo-loop. That is we have at most $|V|$ vertices on the main path after excluding the marked leaf. For each vertex $p$ on the main path other than the leaf, the subtree whose root is the other child of $p$ than the child on the main path is inner pseudo-loop free. Such a subtree has the depth at most $|V|-1$ and therefore has at most $2^{|V|-1}$ leaves. In total, we have at most $|V|2^{|V|-1}$ leaves after excluding the marked one. 
\end{proof}
\begin{remark}
	The bound $|V|2^{|V|-1}$ may not be a tight bound but we can come up with an example where a pseudo-loop of the rate $\lambda$ must have at least $2^m+1$ leaves after excluding the marked leaf for a set of $m+3$ functions $a,b,v_0,v_1,\dots,v_m$ where $M(a)=(a,b), M(b)=(a,v_0), M(v_0)=(v_1,v_1), M(v_1)=(v_2,v_2), \dots, M(v_{m-1})=(v_m,v_m), M(v_m)=(a,a)$ with $c_a=c_b=c_{v_0}=\dots=c_{v_{m-1}}=0$ and $c_{v_m}=1$. The verification is left to the readers as an exercise. (Hint: The growth rate is $2^m/(2^m+1)$.)
\end{remark}

\section{Growth rate as the solution of a linear program}
\label{sec:equiv_to_linear_prog}
\subsection*{Relation to the original program}
We prove Theorem \ref{thm:equiv_to_linear_prog}.

Let $\theta$ and $\{z_v\}_v$ be a solution to the linear program. We prove the following two claims.

\begin{claim} $g(n)\le n\theta + \max_v z_v$. \end{claim}
\begin{proof}
For each $n$, consider the composition tree corresponding to $g(n)$ and let the label of the root be $v^*$. Let $L$ be the multiset of the labels of the leaves in the composition tree. Since $z_v\ge c_v - \theta$ and $z_v\ge z_u+z_w$ for any $v$ and $M(v)=(u,w)$, we have
\[
	z_{v^*} \ge \sum_{u\in L} (c_u - \theta) = g(n) - n\theta \implies g(n) \le z_{v^*} + n\theta,
\]
which confirms the claim.
\end{proof}

\begin{claim} $g(n) \ge n\theta + O(1)$. \end{claim}
\begin{proof}
We say a function $v$ is decomposable if either (i) $z_v = c_v - \theta$, or (ii) $z_v = z_u + z_w$ (for $M(v)=(u,w)$) and both $u,w$ are decomposable.

Let $G$ be the decomposition graph, which is a directed graph with the vertices being the functions and there is an edge from $v$ to $u$ (resp. $w$) if and only if $z_v=z_u+z_w$ (for $M(v)=(u,w)$) and $w$ (resp. $u$) is decomposable. (Note that the condition for a vertex to have an outward edge is weaker than the condition for a vertex to be decomposable.)

We will show that $G$ contains a cycle.
Assume otherwise, that is we have a partial order between the vertices in $G$ with $u\le v$ if there is an edge $vu$.
Consider $\theta'=\theta-\epsilon$ for a small enough $\epsilon$, we show that there is a solution with $\theta'$ (which contradicts with the minimality of $\theta$).
We first start with all decomposable functions $v$ with $z_v = c_v-\theta$ and increase it to $z'_v = c_v-\theta'$ and gradually increase $z_v$ for decomposable functions $v$ with $z_v=z_u+z_w$ to $z'_v=z'_u+z'_w$. Finally, for those $v$ with an edge $vu$ in $G$ whose $z'_v$ is not established yet, we increase $z_v$ to $z'_v=z_u+z'_w$ with $z_v$ for smaller $v$ in the partial order updated first. Note that we do not need to update $z_v$ twice for any $v$. For the remaining functions $v$ we keep $z'_v=z_v$ and obtain a solution $\{z'_v\}_v$ for $\theta'$.

Now $G$ contains a cycle, say $v_0\to v_1\to \dots \to v_k\to v_0$ with $z_{v_i}=z_{v_{i+1}} + z_{w_{i+1}}$ for $M(v_i)= (v_{i+1}, w_{i+1})$ (and $z_{v_k}=z_{v_0}+z_{w_0}$). Since $z_{v_0}=\left(\sum_{i=0}^k z_{w_i}\right) + z_{v_0}$, the sum $\sum_{i=0}^k z_{w_i}$ is zero.

As each $w_i$ is decomposable, we can construct a composition tree so that the root is labeled $w_i$ and $z_{w_i}$ is the sum of $c_v - \theta$ over all the labels $v$ of the leaves.

We now obtain a pseudo-loop whose main path is the same as the cycle in $G$ and the other branches are the above decomposition trees. This pseudo-loop has rate $\theta$ as the sum of $z_{w_i}$ is zero.

Let the number of leaves excluding the marked leaf be $m$, then for any $n=mp+r$ ($1\le r\le m$), the claim follows from the boundedness of $r$ and
\[
	g(n) \ge mp\theta + O(1). \qedhere
\]
\end{proof}

Theorem \ref{thm:equiv_to_linear_prog} follows from the two claims.

\begin{remark}
	Given a solution of the program, it is possible to construct a pseudo-loop of the growth rate in linear time as in the process of the second claim. The least trivial part is to check if the functions are decomposable. We leave it as an exercise for the readers.
\end{remark}

\subsection*{Relation to the dual program}
We relate the dual program to pseudo-loops of the growth rate. The dual program has $2|V|$ variables $\{x_v, y_v: v\in V\}$ so that for each $v$ we have
\[
	\begin{gathered}
		x_v+y_v=\sum_{u,w:\ M(u)=(v,w)} y_u + \sum_{u,w:\ M(u)=(w,v)} y_u,\\
		x_v\ge 0, \\
		y_v\ge 0,
	\end{gathered}
\]
and the sum of all $x_v$ is
\[
	\sum_v x_v=1.
\]
The object of the program is to maximize
\[
	\sum_v c_v x_v.
\]

The maximum value is the same solution as in the original program, which is the growth rate $\lambda$. We show that a pseudo-loop of the rate $\lambda$ can give a solution to the dual program in linear time of the number of variables. In fact, the transform is more straightforward than the other direction with the original program.

Consider a pseudo-loop with the rate $\lambda$. We let $x'_v$ be the number of leaves labeled $v$ in the tree, and let $y'_v$ be the number of non-leaf vertices labeled $v$. If $v$ is the label of the root, we reduce $x'_v$ by $1$ (not counting the marked leaf). All the variables $x'_v,y'_v$ that have not been assigned any value will be assumed to be zero.

By the structure of the tree, we have
\[
	x'_v+y'_v=\sum_{u,w:\ M(u)=(v,w)} y'_u + \sum_{u,w:\ M(u)=(w,v)} y'_u.
\]

Let $m=\sum_v x'_v$, we set $x_v=x'_v/m$ and $y_v=y'_v/m$ for each $v$. We have $\sum_v x_v=1$, and the object $\sum_v c_v x_v$ is the rate of the pseudo-loop, which is $\lambda$. Such a solution gives the maximum value to the object.

\section{Rate test in quadratic time}
\label{sec:rate_test}
We show that it is possible to test whether a proposed rate $\lambda_0$ is smaller than the actual rate $\lambda$ in quadratic time of the number of functions, which in turn immediately gives an algorithm to find an approximation to the growth rate in $O(-|V|^2\log\epsilon)$ for a given precision $\epsilon$. In other words, we settle Theorem \ref{thm:rate-test} as follows.

At first, if we reduce each starting value by $\lambda_0$, then the growth rate is reduced by $\lambda_0$. Therefore, to check $\lambda_0<\lambda$ we just need an algorithm to check if the growth rate of a system is positive. In other words, the question is whether there exists a pseudo-loop of positive rate.
We show that it is in turn equivalent to the existence of a function $v$ not having a tree rooted by label $v$ of maximum value $z_v$ (regardless of the number of leaves). This equivalence will be verified after presenting the following algorithm, which gives maximum values $z_v$ in case there are such values.

\begin{quote}
	\emph{Algorithm}: For each $v$, initiate $z_v = c_v$. We repeat the following process as long as there is a variable $v$ still having the initial value and $z_v < z_u + z_w$ for $M(v) = (u,w)$:
	\begin{itemize}
		\item  Update $z_v$ by the new better value $z_u+z_w$ and mark $z_v$ as a variable depending on $z_u, z_w$ in the sense that any further improvement on $z_u$ or $z_w$ will be directly followed by an improvement on $z_v$.
		\item Make a consequence of improvements on variables that directly or indirectly depend on $z_v$. If $z_v$ is itself a variable among those variables depending on $z_v$, then we stop the iteration and conclude $\lambda > 0$ right away.
	\end{itemize}
	If we finish without concluding $\lambda > 0$, then we conclude otherwise $\lambda\le 0$.
\end{quote}

The process can be done in $O(|V|^2)$ time since the second step in each iteration is a finite process of $O(|V|)$ time, as in the verification of the algorithm below.

We show that each $z_v$ from our algorithm gives the largest possible value over all the compositions trees rooted by $v$ without any inner pseudo-loop.
We reason by induction on the height of trees. Consider a tree $T^*_v$ rooted by $v$ with the maximum value over the trees without any inner pseudo-loop. It means no other occurrence of $v$ other than the root. Note that this tree may be different from the tree $T_v$ produced by our algorithm (due to the order we consider the functions). If $T^*_v$ is only a single vertex $v$, then its value is $c_v$. Our algorithm gives this value in the first place and the value of $T_v$ will never be decreased during the course. Suppose all other functions $v'$ in the tree $T^*_v$ than the root $v$ have their trees $T_{v'}$ produced by the algorithm attaining their maximum values. Since $T_v$ is the tree of two subtrees $T_u, T_w$, whose values are maximum due to the induction hypothesis, the value of $T_v$ is also the maximum value for $v$.

It means if there is no pseudo-loop of positive rate, the values produced by the algorithm are also the maximum values of the trees rooted by the functions. 

On the other hand, if there is any pseudo-loop of positive rate, our algorithm also detects a pseudo-loop of positive rate. In this case, $g(n)$ is unbounded. Consider a minimal composition tree giving a value larger than any $z_v$ given by our algorithm (minimality in the sense that no subtree has such a property).
Each branch of the root should give the value at most the value given by our algorithm due to the minimality of the composition tree. 
Suppose the algorithm stops without recognizing any pseudo-loop. Let $v$ be the label of the root. If $v$ is already marked as being dependent on any improvement of $u,w$ ($M(v)=(u,w)$), then we have a contradiction as $z_v < z_u + z_w$. If the dependency has not been established, then our algorithm has not finished yet, as we still have $z_v < z_u + z_w$ and another iteration should be proceeded. In either case, we have a contradiction.

As a matter of time complexity, we show that for the terminating condition in each iteration, we only need to check for $z_v$ but not any other $z_u$ whether that variable depends the improvement of $z_v$ for the turn $z_v$ is updated.
Initially, there is no pseudo-loop in the composition trees corresponding to all $z_v$. Suppose the same situation before a given iteration.
The reason for that lack is due to a missing edge of dependence. Therefore, if there is a pseudo-loop after updating $z_v$, it must be a pseudo-loop involving $v$ when only two new dependencies $v\to u$ and $v\to w$ are introduced as the missing edges. Also, before reaching again $v$ in case of a pseudo-loop, we do not have to check for other pseudo-loops when updating variables depending on $v$ as they do not exist. The second step of the iteration can be done easily with a queue in $O(|V|)$ time. It follows that the whole algorithm takes $O(|V|^2)$ time since the outmost loop is iterated at most $|V|$ times.

We have verified the validity of the algorithm by showing that the algorithm either stops in the middle and concludes the existence of a pseudo-loop of a positive rate ($\lambda > 0$), or finishes and gives the trees of the maximal values ($\lambda \le 0$).

\begin{remark}
	The best value obtained by the algorithm for the system whose starting values are reduced by $\lambda_0$ is also a solution of $z_v$ with a fixed $\theta=\lambda_0$ to System \eqref{eq:ineq_sys}. Of course, a solution only exists when $\lambda_0\ge \lambda$.
\end{remark}

\section{A cubic time algorithm to find the precise value of the growth rate}
\label{sec:cubic-time}
This section combines the results of Theorem \ref{thm:maximum} and Theorem \ref{thm:rate-test} to give a cubic time algorithm computing the growth rate precisely provided that the starting values are rational. In other words, we settle Theorem \ref{thm:cubic-time} as follows.

At first, we can assume that the starting values are not just rational but all integers, otherwise we can scale the starting values by an appropriate factor. By Theorem \ref{thm:maximum}, the growth rate of a system is the rate of a pseudo-loop without any removable inner pseudo-loop, which is of the form $a/b$ where $b$ is an integer at most $|V|2^{|V|-1}$. By the assumption that the starting values are integers, the numerator $a$ is also an integer and the rates $a_1/b_1$ and $a_2/b_2$ of two pseudo-loops without any removable inner pseudo-loop are either equal or at least $1/B^2$ apart where $B=|V|2^{|V|-1}$. It means we can stop the binary search with the quadratic time rate test algorithm in Theorem \ref{thm:rate-test} whenever the interval is small enough, in particular less than $1/B^2$. This interval contains only one fraction whose denominator is at most $B$, which is the growth rate. Given the interval, we can find this precise value of the growth rate using the Farey sequence in linear time of $|V|$, which is dominated by the time finding the interval, which is $O(|V|^2\log (B^2))=O(|V|^3)$. In fact, instead of taking the middle value in each iteration of the binary search, one can take the mediant as in the process of the Farey sequence and avoid applying the Farey sequence in the end. However, it does not change the cubic time of the algorithm. The algorithm can be seen as a nice combination of the binary search, the Farey sequence and some insights of the problem.

\begin{remark}
	The approach does not apply when the nature of the starting values is more complicated than rational numbers, e.g. transcendental numbers $e,\pi,\dots$. One can approximate these numbers by rationals and then recover the coefficients (the number of leaves with the corresponding label over the total number of leaves) from the estimated growth rate, however, it may take an exponential time for the recovery. The problem in this case seems to ask for a more direct solution than finding the value by the binary search. 
\end{remark}

\section{A proof of the limit using Fekete's lemma provided the dependency graph is connected}
\label{sec:proof-by-fekete}
Suppose the dependency graph is connected, this section provides a simple proof of the limit $\lambda$. It is interesting to apply Fekete's lemma here, as our problem itself can be seen as a variant of Fekete's lemma.

If there is an edge $vu$ with $M(v)=(u,w)$, then
\[
	v(n) \ge u(n-1) + c_w.
\]

It follows that if the distance from $v$ to $u$ is $d_{v,u}$, then
\[
	v(n) \ge u(n-d_{v,u}) + \alpha_{v,u},
\]
for some constant $\alpha_{v,u}$.

Consider a function $v$ with $M(v)=(u,w)$. For any $m,n$ large enough, we have
\[	
	v(m+n)\ge u(m)+w(n) \ge v(m-d_{u,v}) + \alpha_{u,v} + v(n-d_{w,v}) + \alpha_{w,v},
\]
where the constants $d_{u,v}, d_{w,v}$ are valid because the dependency graph is connected.

Adding to both sides $\alpha_{u,v}+\alpha_{w,v}$ and shifting the sequence by $d_{u,v}+d_{w,v}$ steps back, we have
\begin{equation*}
	\begin{multlined}
		v(m+n-d_{u,v}-d_{w,v}) + \alpha_{u,v}+\alpha_{w,v} \ge v(m-d_{u,v}-d_{w,v}) + \alpha_{u,v}+\alpha_{w,v} \\ 
		+ v(n-d_{u,v}-d_{w,v}) + \alpha_{u,v}+\alpha_{w,v}.
	\end{multlined}
\end{equation*}

Let $v'(n)=v(n-d_{u,v}-d_{w,v}) + \alpha_{u,v}+\alpha_{w,v}$, we can see that $v'(n)$ is a superadditive sequence. By Fekete's lemma, $v'(n)/n$ converges. It follows that $v(n)/n$ converges to the same limit. The convergence of $g(n)/n$ follows. (Note that it is still possible to apply Fekete's lemma to a sequence whose some beginning elements are not defined, e.g. by simply assigning small enough values to those elements.)

\begin{remark}
	The approach still works when we replace the equality in Equation \eqref{eq:setting} in the introduction by the inequality $v(n) \ge \max_{1\le m\le n-1} u(n-m) + w(m)$. However, the limit does not necessarily hold when the dependency graph is not connected, as pointed out in the introduction.
\end{remark}

\section*{Acknowledgement}
The author would like to thank G\"unter Rote for his suggestion to the linear program, the relation to Petri nets/pebble games and other helpful comments on this paper, and the anonymous reviewer for suggesting that the dual program may be also interesting.

\bibliographystyle{unsrt}
\bibliography{gor}
\end{document}